\documentclass{amsart}
\usepackage{etex}
\usepackage[a4paper]{geometry}
\usepackage{amssymb,amsfonts,amsmath,amsthm}
\usepackage{comment} 
\usepackage[all,arc]{xy}
\usepackage{enumerate}
\usepackage{mathrsfs,mathtools}
\usepackage{todonotes,booktabs}
\usepackage{stmaryrd}
\usepackage{marvosym}

\usepackage{tikz}
\usepackage{tikz-cd}
\usetikzlibrary{trees}
\usetikzlibrary[shapes]
\usetikzlibrary[arrows]
\usetikzlibrary{patterns}
\usetikzlibrary{fadings}
\usetikzlibrary{backgrounds}
\usetikzlibrary{decorations.pathreplacing}
\usetikzlibrary{decorations.pathmorphing}
\usetikzlibrary{positioning}
	
\def\biblio{\bibliography{duality}\bibliographystyle{alpha}}

\usepackage{xcolor} 
\usepackage{graphicx}

\usepackage[pagebackref]{hyperref}

\definecolor{dark-red}{rgb}{0.5,0.15,0.15}
\definecolor{dark-blue}{rgb}{0.15,0.15,0.6}
\definecolor{dark-green}{rgb}{0.15,0.6,0.15}
\hypersetup{
    colorlinks, linkcolor=dark-red,
    citecolor=dark-blue, urlcolor=dark-green
}
\newcommand{\iHom}{\underline{\operatorname{Hom}}}

\DeclareMathOperator{\cW}{\mathcal{W}}

\renewcommand*{\backref}[1]{}
\renewcommand*{\backrefalt}[4]{%
  \ifcase #1 %
No citations.
  \or
(cit. on p. #2).%
  \else
(cit on pp. #2).%
  \fi%
}
\usepackage{subfiles}

\newtheorem{thm}{Theorem}[section]
\newtheorem{cor}[thm]{Corollary}
\newtheorem{prop}[thm]{Proposition}
\newtheorem{lem}[thm]{Lemma}

\newtheorem{csh}[thm]{Algebraic chromatic splitting hypothesis}

\theoremstyle{definition}
\newtheorem{defn}[thm]{Definition}

\newtheorem{ex}[thm]{Example}

\theoremstyle{remark}
\newtheorem{rem}[thm]{Remark}

\theoremstyle{theorem}
\newtheorem*{thm*}{Theorem}

\DeclareMathOperator{\Sp}{Sp}
\DeclareMathOperator{\Hom}{Hom}

\DeclareMathOperator{\colim}{colim}

\DeclareMathOperator{\cC}{\mathcal{C}}
\DeclareMathOperator{\cD}{\mathcal{D}}

\DeclareMathOperator{\cS}{\mathcal{S}}

\DeclareMathOperator{\cV}{\mathcal{V}}
\DeclareMathOperator{\Id}{\mathrm{Id}}

\DeclareMathOperator{\Spec}{Spec^h}
\DeclareMathOperator{\Mod}{Mod}
\DeclareMathOperator{\Stable}{Stable}

\DeclareMathOperator{\StMod}{StMod}

\DeclareMathOperator{\Loc}{Loc}

\DeclareMathOperator{\Thick}{Thick}

\newcommand{\Q}{\mathbb{Q}}

\newcommand{\bE}{\mathbb{E}}

\DeclareMathOperator{\supp}{supp}
\DeclareMathOperator{\Coloc}{Coloc}

\newcommand{\kos}[2]{{#1}/\!\!/{#2}}

\newcommand{\xr}{\xrightarrow}

\newcommand{\Z}{\mathbb{Z}}

\renewcommand{\frak}{\mathfrak}

\DeclareMathOperator{\Inj}{Inj}
\DeclareMathOperator{\tors}{tors}

\newcommand{\fp}{\mathfrak{p}}
\newcommand{\fq}{\mathfrak{q}}
\newcommand{\fr}{\mathfrak{r}}

\newcommand{\recollement}[5]{
\xymatrix{{#1} \ar[r]|-{#2} & #3 \ar[r]|-{#4} \ar@<1ex>[l]^-{{#2}_!} \ar@<-1ex>[l]_-{{#2}^*} & #5, \ar@<1ex>[l]^-{{#4}!} \ar@<-1ex>[l]_-{{#4}^*}
}}
\let\lim\relax

\DeclareMathOperator{\lim}{lim}
\newcommand{\mm}{/\!\!/}
\newcommand{\bI}{\mathbb{I}}
\DeclareMathOperator{\op}{op}
\newcommand{\cZ}{\mathcal{Z}}

\newcommand{\cM}{\mathcal{M}}
\newcommand{\bW}{\mathbb{W}}
\newcommand{\F}{\mathbb{F}}

\title{The algebraic chromatic splitting conjecture for Noetherian ring spectra}
\author{Tobias Barthel}
\author{Drew Heard}
\author{Gabriel Valenzuela}
\date{\today}

\begin{document}

\begin{abstract}
We formulate a version of Hopkins' chromatic splitting conjecture for an arbitrary structured ring spectrum $R$, and prove it whenever $\pi_*R$ is Noetherian. As an application, these results provide a new local-to-global principle in the modular representation theory of finite groups. 
\end{abstract}

\maketitle

\def\biblio{}

\section{Introduction}

In his seminal talk  \cite{hopkins_global}, Hopkins presents the global structure of the stable homotopy category in parallel to the structure of the derived category $\cD_R$ of a Noetherian commutative ring $R$. In both cases, the thick subcategories of compact objects are classified in terms of a support theory, which in turn is based on a spectrum of certain prime objects. In the algebraic case, Neeman~\cite{neemanchromtower} shows that this spectrum can be taken to be the Zariski spectrum $\mathrm{Spec}(R)$ of prime ideals in $R$, while the corresponding result in homotopy theory has been worked out previously by Devinatz, Hopkins, and Smith \cite{nilpotence1,nilpotence2}.

Based on earlier work of Hovey, Palmieri, and Strickland \cite{hps_axiomatic}, Benson, Iyengar, and Krause \cite{benson_local_cohom_2008,benson_stratifying_2011} subsequently developed a framework for studying both the global and local structure of general triangulated categories in terms of local cohomology and local homology functors. This provides a conceptual approach to formulating and proving many of the fundamental results in stable homotopy theory in various other contexts, as for example modular representation theory. 

An important open question about the stable homotopy category is how its indecomposable pieces assemble locally; this is the content of Hopkins' chromatic splitting conjecture. If correct, it would provide a fundamental local-to-global principle in stable homotopy theory. Given the general framework mentioned above, we may thus formulate the analogous problem in an arbitrary triangulated category equipped with a support theory. 

The goal of this paper is to study this question for the category $\Mod_R$ of module spectra over a structured ring spectrum $R$. To this end, we need to work in the context of the companion paper \cite{bhv2}, which in turn is based on \cite{benson_local_cohom_2008} and \cite{bhv}. If $\cV \subseteq \Spec(\pi_*R)$ is a specialization closed subset, there is a quadruple of functors $(\Gamma_{\cV},L_{\cV},\Delta_{\cV},\Lambda_{\cV})$ on $\Mod_R$ (see \cite[Thm.~3.9]{bhv2}) which captures the part of $\Mod_{R}$ supported on $\cV$. 

In analogy to the arithmetic pullback square displayed on the left below, we construct a homotopy pullback square which describes how a compact $\fp$-local $R$-module spectrum $M$ is assembled from its pieces $\Lambda_{\frak p}M$ and $L_{\cV(\frak p)}M$:
\begin{equation}\label{eq:chromsquare}
\xymatrix{\Z \ar[r] \ar[d] & \prod_{p}\Z_{p} \ar[d] & M \ar[d] \ar[r] & \Lambda_{\frak p}M \ar[d] \\
\Q \ar[r] & \Q \otimes \prod_{p}\Z_p & L_{\cV(\frak p)}M \ar[r]_-{\iota_{\fp}(M)} & L_{\cV(\frak p)}\Lambda_{\frak p}M.}
\end{equation}
Here, the gluing is controlled by the map $L_{\cV(\frak p)}M \to L_{\cV(\frak p)}\Lambda_{\frak p}M$. The algebraic chromatic splitting hypothesis for $R$ and $\fp$ states, informally speaking, that this assembly process is as simple as possible without being trivial. 

 Throughout this document we use the term Noetherian ring spectrum to refer to an $\bE_{\infty}$-ring spectrum $R$ with $\pi_*R$ graded Noetherian. The main theorem of this paper verifies that the algebraic chromatic splitting hypothesis holds for every structured Noetherian ring spectrum and every module $M$ of type $\fp'$, a condition introduced in \protect \MakeUppercase {D}efinition\nobreakspace \ref {defn:type}. As such, it is the first instance of a systematic result on chromatic splitting. 

\begin{thm*}[\protect \MakeUppercase {T}heorem\nobreakspace \ref {thm:chromaticsplitting}]
Suppose that $R$ is a Noetherian $\bE_{\infty}$-ring spectrum, then the map $\iota_{\fp}(M)$ in \eqref{eq:chromsquare} is split for any adjacent pair of primes $(\fp'\subset \fp)$ and any $\fp$-local $R$-module $M$ of type $\fp'$. 
\end{thm*}

The proof relies on a generalization of a result due to Margolis about the nonexistence of certain phantom maps in the category of spectra. This requires a notion of Brown--Comenetz duality for $\Mod_R$, which is of independent interest. In particular, we explain how Brown--Comenetz duality interacts with local duality in the Gorenstein case, thereby obtaining an analogue of a theorem of Hovey and Strickland for the $K(n)$-local stable homotopy category. 

A guiding example is the $\bE_{\infty}$-ring spectrum $R = C^*(BG,k)$ of cochains on the classifying space of a finite $p$-group $G$ with coefficients in a field $k$ of characteristic $p$. The associated module category $\Mod_{C^*(BG,k)}$ is equivalent to Krause's category $\Stable_{kG}$, which is built from the stable module category of $kG$ and its derived category via a recollement. The structure of $\Mod_{C^*(BG,k)}$ therefore controls a large part of the modular representation theory of $G$. 

The main theorem of Benson, Iyengar, and Krause \cite{bik_finitegroups} establishes a decomposition of $\Stable_{kG}$ in terms of certain minimal localizing subcategories $\Gamma_{\fp}\Stable_{kG}$ parametrized by prime ideals $\fp \in \Spec(H^*(G,k))$. Specialized to this category, the chromatic splitting hypothesis thus describes how finitely generated $G$-representations $M$ are built out of their local cohomology complexes $\Gamma_{\fp}M$. In other words, our results can be interpreted as a local-to-global principle for $kG$-modules in the stable module category, see \protect \MakeUppercase {E}xample\nobreakspace \ref {ex:groups}.

From a more abstract point of view, we find the structural similarities between the stable homotopy category $\Sp$ and algebraic categories like $\Stable_{kG}$ rather remarkable. We hope to return to this point in a future paper; for now, in order to help the reader to translate between the different contexts, we end this introduction with a short dictionary.

\renewcommand{\arraystretch}{1.2}
\begin{center}
    \begin{tabular}{| l | l |}
    \hline 
    Chromatic homotopy theory & Modular representation theory \\ \hline
    $I_n = (p,v_1,\ldots,v_{n-1}), n\ge 0$ & $\fp \in \Spec(H^*(G,k))$ \\ 
    $L_n$ & $L_{\cZ(\fp)}$ \\ 
    $M_n,L_{K(n)}$ & $\Gamma_{\fp},\Lambda_{\fp}$ \\ 
    Gross--Hopkins duality & Benson--Greenlees duality. \\ \hline
    \end{tabular}
\end{center}

\subsection*{Conventions}

Throughout this paper, we will work in the setting of $\infty$-categories as developed in \cite{htt,ha}, and will use the local duality framework described in \cite{bhv,bhv2}. In particular, all constructions will implicitly be assumed to be derived. An $\infty$-category $\cC=(\cC,\otimes)$ is called a stable category if it is a symmetric monoidal stable $\infty$-category compactly generated by dualizable objects and whose monoidal product $\otimes$ commutes with colimits separately in each variable. Writing $A$ for a unit of the stable category $(\cC,\otimes)$, we define the (Spanier--Whitehead) dual of an object $X \in \cC$ by $X^{\vee} = \iHom_{\cC}(X,A)$, where $\iHom_{\cC}$ denotes the internal mapping object of $\cC$; note that, under our assumptions on $\cC$, $\iHom_{\cC}$ exists for formal reasons. 

A full stable subcategory $\cD \subseteq \cC$ of a stable category $\cC$ is thick if it is closed under suspensions, finite colimits, and retracts. For example, the full subcategory $\cC^{\omega} \subset \cC$ of compact objects in $\cC$ is thick. A thick subcategory is said to be localizing if it is also closed under all colimits. We denote the thick subcategory generated by a subcategory $\cS \subseteq \cC$ by $\Thick(\cS)$; the smallest localizing subcategory $\Loc(\cS)$ of $\cC$ containing $\cS$ is defined analogously. Finally, a localizing subcategory $\cM \subseteq \cC$ which does not contain any proper localizing subcategories is called minimal.

All discrete rings $R$ in this paper are assumed to be commutative and graded, and all ring-theoretic notions are implicitly graded. In particular, an $R$-module $M$ refers to a graded $R$-module and we write $\Mod_R$ for the abelian category of discrete graded $R$-modules. Prime ideals in $R$ will be denoted by fraktur letters $\fp,\fq,\fr$ and are always assumed to be finitely generated and homogeneous, so that $\Spec(R)$ refers to the Zariski spectrum of homogeneous finitely generated prime ideals in $R$. 

\subsection*{Acknowledgements}

We would like to thank John Greenlees, Henning Krause, and Hal Sadofsky for helpful discussions, as well as the referee for many useful suggestions and corrections. Moreover, we are grateful to the Max Planck Institute for Mathematics for its hospitality, funding a week-long visit of the third-named author in June 2016. The first-named author was partially supported by the DNRF92.

\section{Chromatic assembly}

\subsection{Recollections on local cohomology and local homology}

Let $R$ be an $\bE_{\infty}$-ring spectrum. A subset $\cV\subseteq \Spec(\pi_*R)$ of prime ideals of $\pi_*R$ is called specialization closed if $\fp \in \cV$ and $\fp \subseteq \fq$ imply $\fq \in \cV$. In \cite{benson_local_cohom_2008}, Benson, Iyengar, and Krause construct a smashing colocalization functor $\Gamma_{\cV}$ corresponding to the subcategory of $\Mod_{R}$ on objects with support in $\cV$. We give a brief summary of the basic theory of these functors, following the approach taken in \cite{bhv2}.

For a given finitely generated prime ideal $\fp \in \cV$, we can form the corresponding Koszul object $R\mm\fp \in \Mod_R^{\omega}$ by iteratively coning off the elements in some generating set of $\fp$. While this construction depends on the chosen generators, the thick subcategory $\Thick(R\mm\fp)$ generated by $R\mm\fp$ does not, justifying the notation. Let $\Mod_R^{\mathcal V -\mathrm{tors}} = \Loc(R\mm\fp\mid\fp\in\cV)$ be the localizing subcategory of $\Mod_R$ generated by the Koszul objects $R\mm\fp$ with $\fp \in \cV$. The pair $(\Mod_R,\Mod_R^{\mathcal V -\mathrm{tors}})$ forms a local duality context in the sense of \cite{bhv}, and thus, by \cite[Thm.~2.21 and Cor.~2.26]{bhv} (which in turn relies on \cite[Thm.~3.3.5]{hps_axiomatic}) gives rise to four functors $(\Gamma_{\cV},L_{\cV},\Delta_{\cV},\Lambda_{\cV})$, satisfying:

\begin{thm}\label{thm:localduality}
Suppose $R$ is an $\bE_{\infty}$-ring spectrum and $\cV \subseteq \Spec(\pi_*R)$ is a specialization closed subset of prime ideals. There exists a quadruple $(\Gamma_{\cV},L_{\cV},\Delta_{\cV},\Lambda_{\cV})$ of endofunctors on $\Mod_R$ with the following properties:
\begin{enumerate}
	\item The functor $\Gamma_{\mathcal V}$ is the colocalization with respect to the localizing subcategory $\Loc(R\mm\fp \mid \fp \in \cV)$, and both $L_{\mathcal V}$ and $\Lambda_{\mathcal V}$ are localization functors, the latter with essential image the colocalizing subcategory  $ \Coloc(\kos{R}{\fp} \, \otimes M \mid \fp \in \mathcal{V}, M \in \Mod_R)$.
   Moreover, there are natural cofiber sequences
	\[
	\xymatrix{\Gamma_{\mathcal V} M \ar[r] & M \ar[r] & L_{\mathcal V} M & \text{and} & \Delta_{\mathcal V} M  \ar[r] & M \ar[r] & \Lambda_{\mathcal V} M}
	\]
	for all $M \in \Mod_R$. 
	\item Both $\Gamma_{\mathcal V}$ and $L_{\cV}$ are smashing, i.e., $\Gamma_{\mathcal V}(X) \simeq X \otimes \Gamma_{\mathcal V}(R)$ for all $X$ and similarly for $L_{\mathcal V}$. Moreover, $L_{\mathcal V}$ preserves compact objects and we denote its essential image by $\Mod_R^{\cV-\mathrm{loc}}$.
	\item The functors $\Gamma_{\mathcal V}$ and $\Lambda_{\mathcal V}$ induce mutually inverse equivalences between their images
	\[
	\xymatrix{\Mod_R^{\mathcal V -\mathrm{tors}} = \mathrm{Im}(\Gamma_{\cV}) \ar@<0.5ex>[r]^-{\Lambda_{\cV}} &  \mathrm{Im}(\Lambda_{\cV}) = \Mod_R^{\mathcal V -\mathrm{comp}} \ar@<0.5ex>[l]^-{\Gamma_{\cV}}.}
	\]
	 Moreover, there are natural equivalences of functors
	\[
	\xymatrix{\Lambda_{\mathcal V} \Gamma_{\mathcal V} \ar[r]^-{\sim} & \Lambda_{\mathcal V} & \Gamma_{\mathcal V} \ar[r]^-{\sim} & \Gamma_{\mathcal V} \Lambda_{\mathcal V}.}
	\]
	\item The functors $(\Gamma_{\mathcal V},\Lambda_{\mathcal V})$ form an adjoint pair, so that we have a natural equivalence 
	\[
	\xymatrix{\Hom_R(\Gamma_{\mathcal V} X,Y) \simeq \Hom_R(X,\Lambda_{\mathcal V} Y)}
	\]
	for all $X,Y \in \Mod_R$, where $\Hom_R(X,Y)$ denotes the $R$-module function spectrum between $R$-modules $M$ and $N$. Similarly, $L_{\mathcal V}$ is left adjoint to $\Delta_{\mathcal V}$. 
	\item There is a homotopy pullback square of functors 
	\[
	\xymatrix{\Id \ar[r] \ar[d] & \Lambda_{\cV} \ar[d] \\
	L_{\cV} \ar[r] & L_{\cV}\Lambda_{\cV},}
	\]
	which is usually referred to as the (chromatic) fracture square. 
\end{enumerate}
\end{thm}

Associated to a given prime ideal $\fp\in\Spec(\pi_*R)$, there are two distinguished specialization closed subsets of prime ideals, namely
\[
\xymatrix{\cZ(\fp) = \{\fq\in\Spec(\pi_*R)\mid \fq \nsubseteq \fp\} & \text{and} & \cV(\fp) = \{\fq\in\Spec(\pi_*R)\mid \fp \subseteq \fq\}.}
\]
The localization functor $L_{\cZ(\fp)}=(-)_{\fp}$ corresponds to algebraic $\fp$-localization, see \cite[Prop.~6.0.7]{hps_axiomatic} and \cite[Thm.~4.7]{benson_local_cohom_2008}, i.e., it is characterized by the property that $(\pi_*M)_{\fp} \cong \pi_*(L_{\cZ(\fp)}M)$. We also have a formula for the localization functors associated to $\cV(\fp)$. To describe this, we remind the reader that for $a \in \pi_*R$ homogeneous of degree $-d$ we define the Koszul object $\kos{R}{a^k}$ by the cofiber sequence 
\begin{equation}\label{eq:koszulhigher}
R \xr{a^k} \Sigma^{kd} R \to \kos{R}{a^k}.
\end{equation}
It is evident from this description that $\kos{R}{a^k}$ is self-dual up to a shift. 

More generally, for an ideal $\fp = (p_1,\ldots,p_n)$, we define\footnote{We warn the reader that $\kos{R}{\fp^{(s)}}$ is not the derived quotient with respect to the ideal $\fp^s$.} $\kos{M}{\fp^{(s)}} = M \otimes \kos{R}{p_1^s} \otimes \cdots \otimes \kos{R}{p_n^s}$.

\begin{lem}\label{lem:koszulformula}
For any prime ideal $\fp = (p_1,\ldots,p_n) \in \Spec(\pi_*R)$ and all $M \in \Mod_R$, there are natural equivalences 
\[\Gamma_{\cV(\fp)}M \simeq \colim_s (\Sigma^{-n} R\mm\fp^{(s)} \otimes M), \quad \text{and} \quad \Lambda_{\cV(\frak p)}M \simeq \lim_s  ( \Sigma^{sd}\kos{R}{\frak p^{(s)}} \otimes M)\] where $-d = |p_1| + \ldots + |p_n|$.
\end{lem}

The effect of $L_{\cV}$ on homotopy groups for arbitrary specialization closed subsets $\cV \subseteq \Spec(\pi_*R)$ is more complicated, but we expect these can be understood via a local cohomology spectral sequence as described in \cite[Rem.~3.15]{bhv2}.

Following \cite{hps_axiomatic} and \cite{benson_local_cohom_2008}, we can then build functors that isolate the part of $\Mod_R$ supported at $\fp$. 

\begin{defn}
For a prime ideal $\fp \in \Spec(\pi_*R)$, define endofunctors on $\Mod_R$ by
\[
\xymatrix{\Gamma_{\fp} = \Gamma_{\cV(\fp)}L_{\cZ(\fp)} & \text{and} & \Lambda_{\fp} = \Lambda_{\cV(\fp)}\Delta_{\cZ(\fp)}.} 
\]
We denote the essential image $\Gamma_{\fp}\Mod_R = \Mod_{R_{\fp}}^{\fp-\text{tors}}$ and $\Lambda_{\fp}\Mod_R = \Mod_{R_{\fp}}^{\fp-\text{comp}}$. Moreover, the support of any $X\in\Mod_R$ is defined as $\supp_R X = \{ \fp \in \Spec(\pi_*R) \mid \Gamma_{\fp}X \ne 0 \}$.
\end{defn}

Note that, as the composite of smashing functors, $\Gamma_{\fp}$ is smashing, but this is not the case for $\Lambda_{\fp}$. As an easy consequence of \protect \MakeUppercase {T}heorem\nobreakspace \ref {thm:localduality}(4), $\Gamma_{\fp}$ is left adjoint to $\Lambda_{\fp}$ for any $\fp \in \Spec(\pi_*R)$; for more details about these functors, we refer to \cite{benson_local_cohom_2008} and \cite{bhv2}. From now on, we will implicitly assume that the letters $\fp,\fq,\fr$ and so on refer to prime ideals in the homotopy groups of the ring spectrum under consideration. 


Recall that the local-to-global principle for $\Mod_R$ states that the inclusion functor
\[
\Loc(\Gamma_{\fp}R\mid \fp \in \Spec(\pi_*R)) \subseteq \Loc(R),
\]
is an equivalence, see \cite[Sec.~3]{benson_stratifying_2011}. In particular, it implies that the natural functor
\begin{equation}\label{eq:localtoglobal}
\xymatrix{\prod_{\fp}\Gamma_{\fp}\colon \Mod_R \ar[r] & \prod_{\fp}\Mod_R}
\end{equation}
is conservative, where the product is indexed by all prime ideals $\fp$. It holds in particular for any $\bE_{\infty}$-ring spectrum with $\pi_*R$ (graded) Noetherian. To see this, it suffices to show that it is an equivalence in the homotopy category. Since $\Mod_R$ is compactly generated by its tensor unit $R$, this follows from \cite[Thm.~7.2]{benson_stratifying_2011}.

\subsection{Adjacent prime ideals}

The goal of this subsection is to establish a relation between the localization and colocalization functors of the previous subsection which will be needed in the proof of the main theorem. 
For any subset $\cV \subseteq \Spec(\pi_*R)$, let $\cV^{\complement} =\Spec(\pi_*R) \setminus \cV $ be the complement. Note that $\cV$ is specialization closed if and only if $\cV^{\complement}$ is generalization closed, i.e., if $\fp \in \cV^{\complement}$ and $\fq \subseteq \fp$ then $\fq \in \cV^{\complement}$. For the following, see~\cite[Lem.~2.4]{benson_stratifying_2011}. 
\begin{lem}\label{lem:supportloccohom}
For any prime ideal $\fp$ and for any specialization closed subset $\cV \subseteq \Spec(\pi_*R)$, the counit map of $\Gamma_{\cV}$ induces natural equivalences
\[
\Gamma_{\fp}\Gamma_{\cV} \simeq 
\begin{cases}
\Gamma_{\fp} & \text{if }\fp \in \cV \\
0 & \text{otherwise}.
\end{cases}
\]
\end{lem}

The next result also appears in~\cite[Prop.~6.1]{benson_local_cohom_2008}, but we give an alternative argument. 

\begin{lem}\label{lem:comparisonloccohom}
If $\cV,\cW \subseteq \Spec(\pi_*R)$ are specialization closed subsets, then the following hold:
\begin{enumerate}
	\item $\Gamma_{\cV}\Gamma_{\cW} \simeq \Gamma_{\cV \cap \cW} \simeq \Gamma_{\cW}\Gamma_{\cV}$.
	\item $L_{\cV}L_{\cW} \simeq L_{\cV \cup \cW} \simeq L_{\cW}L_{\cV}$.
\end{enumerate}
\end{lem}
\begin{proof}
We will only prove (1), and leave the similar proof of (2) for the reader. Since the claim is symmetric in $\cV$ and $\cW$, it suffices to verify that $\Gamma_{\cV}\Gamma_{\cW} \simeq \Gamma_{\cV \cap \cW}$. As $\cV \cap \cW \subseteq \cW$, there exists a natural map $\Gamma_{\cV \cap \cW} \to \Gamma_{\cW}$, which induces a natural transformation $\Gamma_{\cV \cap \cW} \simeq \Gamma_{\cV}\Gamma_{\cV \cap \cW} \to \Gamma_{\cV}\Gamma_{\cW}$ by \cite[Lem.~3.4]{benson_local_cohom_2008}. By Lemma\nobreakspace \ref {lem:supportloccohom}, we have equivalences
\begin{align*}
\Gamma_{\fp}\Gamma_{\cV}\Gamma_{\cW} & \simeq 
\begin{cases}
\Gamma_{\fp}\Gamma_{\cW} & \text{if } \fp \in \cV \\
0 & \text{otherwise}
\end{cases}\\
& \simeq
\begin{cases}
\Gamma_{\fp} & \text{if } \fp \in \cV \cap \cW \\
0 & \text{otherwise},
\end{cases}
\end{align*}
for any $\fp \in \Spec(\pi_*R)$. The same calculation works for $\Gamma_{\cV \cap \cW}$, so the result follows from the local-to-global principle \eqref{eq:localtoglobal}.
\end{proof}
\begin{prop}\label{prop:localdecomp}\sloppy
  Let $\fp \ne \fp'$ be prime ideals in $\Spec(\pi_*R)$, then there is a natural equivalence of functors $$\xymatrix{L_{\cV(\fp)}L_{\cZ(\fp)}\Gamma_{\cV(\fp')} \ar[r]^-{\simeq}& \Gamma_{\fp'}L_{\cZ(\fp)}}$$
    if and only if there is no prime ideal $\fq$ such that $\fp' \subsetneq \fq$ and $\fq \subsetneq \fp$. Furthermore, there is an equivalence
    \[
      \Gamma_{\fp'}L_{\cZ(\fp)} \simeq
    \begin{cases}
    \Gamma_{\fp'} & \text{ if } \fp' \subseteq \fp \\
    0, & \text{ otherwise.}
    \end{cases}
    \]
\end{prop}
\begin{proof}
  For any $\fp, \fp' \in \Spec(\pi_*R)$, the unit maps provide natural transformations
  \[
\xymatrix{ L_{\cV(\fp)}L_{\cZ(\fp)}\Gamma_{\cV(\fp')} \ar[r]^-{\eta} &L_{\cZ(\fp')}L_{\cV(\fp)}L_{\cZ(\fp)}\Gamma_{\cV(\fp')} & \ar[l]_-{\tilde \eta} L_{\cZ(p')}L_{\cZ(\fp)}\Gamma_{\cV(\fp')} \simeq \Gamma_{\fp'}L_{\cZ(\fp)}.} 
  \]
  We first claim that $\tilde \eta$ is an equivalence whenever $\fp \ne \fp'$. It is enough to show that the fiber of $\tilde \eta$ is zero. But this fiber is equivalent to $L_{\cZ(\fp')}\Gamma_{\cV(\fp)}L_{\cZ(\fp)}\Gamma_{\cV(\fp')} \simeq \Gamma_{\fp'}\Gamma_{\fp}$. By \cite[Lem.~4.26]{bhv2} this is zero unless $\fp = \fp'$. 

  Now consider the fiber of $\eta$, which is equivalent to $\Gamma_{\cZ(\fp')}L_{\cV(\fp)}L_{\cZ(\fp)}\Gamma_{\cV(\fp')}$. To show this is zero it is enough to show that the support is empty when evaluated on $R$ \cite[Thm.~2]{benson_local_cohom_2008}. Using Lemma\nobreakspace \ref {lem:comparisonloccohom} the fiber (evaluated at $R$) is equivalent to $\Gamma_{\cZ(\fp') \cap \cV(\fp')} L_{\cZ(\fp) \cup \cV(\fp)}R$. Using \cite[Thm.~5.6]{benson_local_cohom_2008} this has support $\cZ(\fp') \cap \cV(\fp')\cap (\cZ(\fp) \cup \cV(\fp))^{\complement}$. Now $(\cV(\fp) \cup \cZ(\fp))^{\complement} = \{\fq \mid \fq \subsetneq \fp\}$ and $\cZ(\fp') \cap \cV(\fp') = \{ \fq \mid \fp' \subsetneq \fq \}$, and so the support is empty precisely when there is no $\fq' \in \Spec(\pi_*R)$ such that $\fp' \subsetneq \fq$ and $\fq \subsetneq \fp$, as required. 

The second statement follows from the observation that $\Gamma_{\fp'}L_{\cZ(\fp)} \simeq \Gamma_{\fp'}$ if and only $\fp' \in \cZ(\fp)^{\complement}$, which in turn is equivalent to $\fp' \subseteq \fp$; if this is not the case, then $\Gamma_{\fp'}L_{\cZ(\fp)} \simeq 0$.
\end{proof}
This proposition leads naturally to the following definition. 
\begin{defn}\label{defn:adjacent}
A pair $(\fp' \subset \fp)$ of primes in $\Spec(\pi_*R)$ is called adjacent if $\fp' \ne \fp$ and there exists no prime ideal $\fq$ such that $\fp' \subsetneq \fq \subsetneq \fp$.
\end{defn}
The following is an immediate consequence of the definition and \protect \MakeUppercase {P}roposition \ref{prop:localdecomp}. 
\begin{cor}\label{cor:adjacentdecomp}
  For any adjacent pair $(\fp' \subset \fp)$ of primes, there is an equivalence of functors 
\[
\xymatrix{L_{\cV(\fp)}L_{\cZ(\fp)}\Gamma_{\cV(\fp')} \ar[r]^-{\simeq}& \Gamma_{\fp'}.}
\]
\end{cor}
\subsection{The algebraic splitting hypothesis}

In this subsection, we formulate an analogue of Hopkins' chromatic splitting conjecture in chromatic homotopy theory, as described in \cite{hovey_csc}, for arbitrary structured ring spectra $R$.  We will start by constructing an appropriate version of the chromatic fracture square. To this end, first observe that $\Lambda_{\fp}L_{\cZ(\fp)} = \Lambda_{\cV(\fp)}\Delta_{\cZ(\fp)}L_{\cZ(\fp)} \simeq \Lambda_{\cV(\fp)}L_{\cZ(\fp)}$, so  
\begin{equation}\label{eq:localformula}
\Lambda_{\fp}M \simeq \Lambda_{\cV(\fp)}M
\end{equation}
for all $M \in \Mod_{R_{\fp}}$.

\begin{lem}\label{lemp:pcfs}
For any $ M$ in $\Mod_{R_{\frak p}}$ there is a homotopy pullback square 
\[
\xymatrix{
M \ar[d] \ar[r] & \Lambda_{\frak p}M \ar[d] \\
L_{\cV(\frak p)}M \ar[r]_-{\iota_{\fp}(M)} & L_{\cV(\frak p)}\Lambda_{\frak p}M,}
\]
whose horizontal fibers are equivalent to $\Delta_{\fp}M = \Lambda_{\cV(\fp)}\Delta_{\cZ(\fp)}M$. 
\end{lem}
\begin{proof}
This follows from the chromatic pullback square of \protect \MakeUppercase {T}heorem\nobreakspace \ref {thm:localduality}(5) for the local duality context $(\Mod_{R},\cV(\fp))$  combined with \eqref{eq:localformula}.
\end{proof}

The fracture square can be lifted to a categorical decomposition of $\Mod_{R_{\frak p}}$. This makes precise the sense in which the natural transformation $L_{\cV(\frak p)} \to L_{\cV(\frak p)}\Lambda_{\frak p}$ controls the categorical gluing process. 

\begin{prop}
There is a pullback square of stable $\infty$-categories
\[
\xymatrix{\Mod_{R_{\frak p}} \ar[r] \ar[d] &  \Lambda_{\fp}\Mod_{R_{\fp}} \ar[d]^-{L_{\cV(\fp)}} \\
\left (\Mod_{R_{\frak p}}^{\cV(\frak p)-\mathrm{loc}}\right )^{\Delta^1} \ar[r]_-{\pi_1} & \Mod_{R_{\frak p}}^{\cV(\frak p)-\mathrm{loc}},}
\]
in which $\pi_1$ denotes the evaluation at $1 \in [0,1] = \Delta^1$. The left vertical functor sends a module $M \in \Mod_{R_{\frak p}}$ to the bottom map $L_{\cV(\frak p)}M \to L_{\cV(\frak p)}\Lambda_{\frak p}M$ in the fracture square of Lemma\nobreakspace \ref {lemp:pcfs}. 
\end{prop}
\begin{proof}
The pullback square of categories is a special case of \cite[Cor.~2.28]{bhv} applied to the local duality context $(\Mod_{R},\cV(\fp))$ and restricted to $ \Mod_{R_{\frak p}}$. To identify the right upper corner, we note that \eqref{eq:localformula} gives a canonical equivalence
\[
\xymatrix{\Mod_{R_{\frak p}}^{\cV(\frak p)-\mathrm{comp}} \ar[r]^-{\sim} &  \Lambda_{\fp}\Mod_{R_{\fp}}}
\]
of symmetric monoidal stable $\infty$-categories. 
\end{proof}

\begin{defn}\label{defn:type}
A compact $R$-module $M$ is said to be of type $\fp$ if $M \in \Thick(R\mm\fp)$. Equivalently, a compact module $M$ is of type $\fp$ if and only if its support, i.e. the set $\{ \fp \in \Spec (\pi_*R) \mid \Gamma_{\fp} M \ne 0 \}$, is contained in $\cV(\fp)$.
\end{defn}

The next lemma is an analogue of Ravenel's result \cite[Thm.~2.11]{ravenel_localization} saying that (among other things), for a finite spectrum $X \in \Sp$, $K(n)_*(X) =0$ implies $K(m)_*(X)=0$ for all $0 \le m\le n$. However, since the topology of $\Spec(\pi_*R)$ is more complicated, this lemma does not provide an alternative characterization of type. 

\begin{lem}
A compact module $M \in \Mod_R^{\omega}$ satisfies $L_{\cZ(\fq)}M = 0$ if and only if $\Gamma_{\fq}M \simeq 0$. In particular, both conditions holds if $M$ is of type $\fp$ and $\fq \notin \cV(\fp)$.
\end{lem}
\begin{proof}
By \cite[Thm.~6.1.8]{hps_axiomatic} there is an equality of Bousfield classes $\langle \Gamma_{\fq} R \rangle = \langle R_{\fq}\mm \fq \rangle$. Combining this with \cite[Prop.~6.1.7(b)]{hps_axiomatic} we then have, for all $M \in \Mod_R^\omega$,  
\[
\Gamma_{\fq}M \simeq 0 \Longleftrightarrow R_{\fq}\mm \fq \otimes M \simeq 0 \Longleftrightarrow L_{\cZ(\fq)}M \simeq 0.
\]
Finally, if $M$ is of type $\fp$, then $\Gamma_{\fq}M \simeq 0$ for all $\fq \notin \cV(\fp)$. 
\end{proof}

With the notation of Definition\nobreakspace \ref{defn:adjacent} at hand, we may now state the following

\begin{csh}\label{csh}
Suppose $R$ is an $\bE_{\infty}$-ring spectrum. If $(\fp' \subset \fp)$ is an adjacent pair of primes in $\Spec(\pi_*R)$, then $\iota_{\fp}(M)$ is split for any compact $M \in \Mod_{R_{\fp}}$ of type $\fp'$. 
\end{csh}

\begin{rem}
The algebraic chromatic splitting hypothesis given here is inspired by and analogous (albeit not equivalent) to Hopkins' chromatic splitting hypothesis in chromatic homotopy theory, see \cite{hovey_csc}. In its simplest form, the chromatic splitting conjecture asks whether the bottom horizontal map in the chromatic fracture square
\[
\xymatrix{L_nX \ar[r] \ar[d] & L_{K(n)}X \ar[d]\\
L_{n-1}X \ar[r] & L_{n-1}L_{K(n)}X} 
\]
is split, where $X$ is a finite spectrum of type $n-1$, and $L_n$ and $L_{K(n)}$ denote Bousfield localization at Morava $E$-theory $E_n$ and Morava $K$-theory $K(n)$, respectively. Due to computational verification, it is known to be true for $n=0,1,2$, but completely open otherwise. Moreover, there are several refinements of this conjecture, the strongest form of which has been recently disproven by Beaudry \cite{beaudrycsc}.

However, since the corresponding localization functors are constructed differently, the algebraic chromatic splitting hypothesis for $R=S^0$ or $R=L_nS^0$ is not equivalent to the chromatic splitting conjecture, but should instead be considered as an algebraic analogue. This analogy is clearest under the additional assumption that $\Mod_R$ is canonically stratified by $\pi_*R$ in the sense of Benson--Iyengar--Krause~\cite{benson_stratifying_2011}. 
\end{rem}

\begin{rem}\label{rem:localhomvariant}
There is an alternative construction of the local homology functor which is closer to the situation in chromatic homotopy theory, namely $\widetilde{\Lambda}_{\fp} = \Lambda_{\cV(\fp)}L_{\cZ(\fp)}$. As we have seen in \eqref{eq:localformula}, $\Lambda_{\fp}$ and $\widetilde{\Lambda}_{\fp}$ agree on $\fp$-local objects, but they differ in general. As will be clear from the proof of \protect \MakeUppercase {T}heorem\nobreakspace \ref {thm:chromaticsplitting} below, the analogue of the algebraic chromatic splitting hypothesis\nobreakspace \ref {csh} for $\widetilde{\Lambda}_{\fp}$ holds for $\Mod_{R}$, i.e., without imposing any locality conditions. In contrast, $\Lambda_{\fp}$ is right adjoint to $\Gamma_{\fp}$, while this is only true $\fp$-locally for $\widetilde{\Lambda}_{\fp}$. 
\end{rem}

\section{Chromatic splitting}

Throughout this subsection, assume that $R$ is a Noetherian $\bE_{\infty}$-ring spectrum and recall that all prime ideals $\fp \in \Spec(\pi_*R)$ are finitely generated. In fact, it suffices to assume that $R$ has less structure, but we will not need this extra generality here.

\subsection{Phantom maps and Brown--Comenetz duality}

In this subsection, we generalize Brown--Comenetz duality to $\Mod_R$ and its local analogues, and prove a version of Margolis' nonexistence result for phantom maps with target $\bI X$. Furthermore, we describe how Brown--Comenetz duality is related to local duality and (local) Spanier--Whitehead duality. 

\begin{defn}
A map $f \colon X \to Y$ in $\Mod_R$ is called phantom if for all $C \in \Mod_R^\omega$ and maps $g \colon C \to X$ the composite $g \circ f$ is null. 
\end{defn}

For any (discrete) graded commutative ring $A$, denote the full subcategory of graded injective $A$-modules by $\Inj_A \subseteq \Mod_A$. For any $X \in \Mod_R$ and any graded injective module $I \in \Inj_{\pi_*R}$, the exact functor
\[
\xymatrixcolsep{3pc}
\xymatrix{\Mod_R^{\op} \ar[r]^-{\pi_*(X \otimes - )} & \Mod_{\pi_*R}^{\op} \ar[r]^-{\Hom(-,I)} & \Mod_{\pi_*R}}
\]
is representable by an object $\mathbb{I}X$. Note that $\bI R \simeq \bI$ is the lift of $I$ as defined in \cite[Sec.~4.1]{bhv2}. The next lemma is the analogue of a result due to Margolis~\cite[Prop.~5.1.2]{margolis} for the stable homotopy category; his proof generalizes easily.  

\begin{lem}\label{lem:nophantomintobc}
For any $X \in \Mod_R$ and any lift $\bI$ of some injective $\pi_*R$-module $I$, there are no nontrivial phantom maps with target $\bI X$. 
\end{lem}
\begin{proof}
The argument is a straightforward adaptation of Margolis' proof, which we give here for the convenience of the reader. Suppose $f\colon Y \to \bI X$ is a phantom map, and write $Y$ as a filtered colimit $\colim_{J}Y_{\alpha}$ of compact objects $Y_{\alpha}$. Consider the following commutative diagram
\[
\xymatrix{\pi_*\Hom_R(Y,\bI X) \ar[r]^-{\sim} \ar[d] & \Hom_{\pi_*R}^{-*}(\pi_*(X\otimes Y), I) \ar[d]^-{\sim} \\
\lim_J\pi_*\Hom_R(Y_{\alpha},\bI X) \ar[r]_-{\sim} & \lim_J\Hom_{\pi_*R}^{-*}(\pi_*(X\otimes Y_{\alpha}), I).}
\]
The horizontal maps are isomorphisms by the universal property of $\bI X$, while the right vertical map is so because $\pi_*$ commutes with filtered colimits. Therefore, the left vertical map is an isomorphism as well, and the claim follows. 
\end{proof}

Let $I_{\fp}$ be the injective hull of the residue field $\pi_*(R_{\fp})/\fp$ considered as an $R_{\fp}$-module. We write $\bI_{\fp} \in \Mod_{R_{\fp}}$ for the corresponding Brown--Comenetz dual of $R_{\fp}$ and $\phi_M\colon M \to \bI_{\fp}^{2}M$ for the canonical map. The next result gives a sufficient condition for when this map is an equivalence.

\begin{prop}\label{prop:bcduality}
The map $\phi_{M}$ is an equivalence if $\pi_*(M)$ is finitely generated as a module over $\pi_*R_{\fp}/\fp$. In particular, this is the case if $M \in \Thick(R_{\fp}\mm\fp)$. More generally, $\phi_M$ factors through an equivalence $M \simeq \Gamma_{\fp}\bI_{\fp}^{2}M$ for all $M \in \Mod_{R_{\fp}}^{\fp-\tors}$ such that $\pi_*(M \otimes R_{\fp}\mm\fp)$ is finitely generated over $\pi_*R_{\fp}/\fp$.
\end{prop}
\begin{proof}
Suppose $M \in \Mod_{R_{\fp}}$ satisfies the condition of the proposition. Matlis duality and the definition of $\bI_{\fp}$ then provide isomorphisms
\begin{align*}
\pi_*\bI_{\fp}^{2}M & \cong \Hom^{-*}_{\pi_*{R_{\fp}}}(\pi_*(\bI M),I_p) \\
& \cong \Hom_{\pi_*{R_{\fp}}}^{-*}(\Hom^{-*}_{\pi_*R_{\fp}}(\pi_*(M),I_p),I_p) \\
& \cong \pi_*M.
\end{align*}
It is easy to verify that the composite isomorphism is the one induced by $\phi_M$. Furthermore, a standard inductive argument on the number of generators of $\fp$ implies that $\pi_*R_{\fp}\mm\fp$ is finitely generated over $\pi_*(R_{\fp})/\fp$, hence $\phi_M$ is an equivalence for all $M \in \Thick(R_{\fp}\mm\fp)$.

Now assume $M \in \Mod_{R_{\fp}}^{\fp-\tors}$ with the property that $\pi_*(M \otimes R_{\fp}\mm\fp)$ is finitely generated over $\pi_*R_{\fp}/\fp$. Since $M \simeq \Gamma_{\fp}M$, $\phi_M$ factors through a map 
\[
\xymatrix{\phi'_M\colon M \ar[r] & \Gamma_{\fp}\bI_{\fp}^2M,}
\]
so it suffices to prove that $\phi'_M \otimes R_{\fp}\mm\fp \simeq \phi_M \otimes R_{\fp}\mm\fp$ is an equivalence. Since $R_{\fp}\mm\fp$ is compact and thus dualizable in $\Mod_{R_{\fp}}$, the same argument as for \cite[Thm.~10.2(d)]{hovey_morava_1999} shows that $\phi_M \otimes R_{\fp}\mm\fp \simeq \phi_{M \otimes R_{\fp}\mm\fp}$. Therefore, we may assume that $\pi_*M$ is finitely generated over $\pi_*R_{\fp}/\fp$, reducing the claim to the previous case. 
\end{proof}

\begin{cor}\label{cor:nophantom}
For any prime $\fp$, there are no nontrivial phantom maps with target $R_{\fp}\mm\fp$. 
\end{cor}
\begin{proof}
By \protect \MakeUppercase {P}roposition\nobreakspace \ref {prop:bcduality}, there is an equivalence $R_{\fp}\mm\fp \simeq \mathbb{I}_{\fp}^2R_{\fp}\mm\fp$, so the claim follows from Lemma\nobreakspace \ref {lem:nophantomintobc}.
\end{proof}

We finally explain how Brown--Comenetz duality interacts with local duality, at least in the case that $R_{\fp}$ is absolute Gorenstein of shift $\nu$. Recall that this means that there is an (abstract) equivalence $\Gamma_{\fp}R \simeq \Sigma^{\nu}\bI_{\fp}$. For more details on Gorenstein ring spectra, see \cite{greenlees_hi} or \cite{bhv2}. For the following, we let $M^{\vee} = \Hom_{R_{\fp}}(M,R_{\fp})$ denote the dual of $M$ in $\Mod_{R_{\fp}}$. 

\begin{prop}\label{prop:bclocalduality}
Suppose $R_{\fp}$ is absolute Gorenstein of shift $\nu$. For any $M \in \Mod_{R_{\fp}}^{\omega}$, there is an equivalence $\Lambda_{\fp}M \simeq \Sigma^{\nu}\bI_{\fp}(\Gamma_{\fp}M^{\vee})$. In particular, all phantom maps with target $\Lambda_{\fp}M$ must be trivial.  
\end{prop}
\begin{proof}
The Gorenstein condition combined with \protect \MakeUppercase {T}heorem\nobreakspace \ref {thm:localduality} gives isomorphisms 
\begin{align*}
\pi_*\Hom_R(X,\bI_{\fp}\Gamma_{\fp}M^{\vee}) & \cong \Hom^{-*}_{\pi_*R}(\pi_*(X \otimes \Gamma_{\fp}M^{\vee}), I_{\fp}) \\
& \cong  \pi_*\Hom_R(X \otimes \Gamma_{\fp}M^{\vee}, \bI_{\fp}) \\
& \cong \pi_*\Hom_R(\Gamma_{\fp}X \otimes M^{\vee}, \bI_{\fp}) \\
& \cong  \pi_*\Hom_R(X, \Lambda_{\fp}(M \otimes \bI_{\fp})) \\
& \cong  \pi_*\Hom_R(X, \Lambda_{\fp}(M \otimes \Sigma^{-\nu}\Gamma_{\fp}R)) \\
& \cong  \pi_*\Hom_R(X, \Lambda_{\fp}\Sigma^{-\nu}M)
\end{align*}
for any $X \in \Mod_{R}$. Consequently, we get an equivalence $\Lambda_{\fp}M \simeq \Sigma^{\nu}\bI_{\fp}(\Gamma_{\fp}M^{\vee})$. The second part of the claim then follows from Lemma\nobreakspace \ref {lem:nophantomintobc}.
\end{proof}

\begin{cor}\label{cor:gorensteinphantom}
If $R_{\fp}$ is absolute Gorenstein of some shift $\nu$ and $M \in \Mod_R^{\omega}$, then all phantom maps with target $\widetilde{\Lambda}_{\fp}M$ are trivial.
\end{cor}
\begin{proof}
As noted before, $\widetilde{\Lambda}_{\fp}M \simeq \widetilde{\Lambda}_{\fp}M_{\fp} \simeq {\Lambda}_{\fp}M_{\fp}$. Since $M_{\fp} \in \Mod_{R_{\fp}}^{\omega}$ is compact, Proposition\nobreakspace \ref{prop:bclocalduality} applies.
\end{proof}

\begin{rem}
We could also make the stronger global assumption that $R$ rather than $R_{\fp}$ is Gorenstein. In this case, there is an additional shift by $d$ appearing in the formula of \protect \MakeUppercase {P}roposition\nobreakspace \ref {prop:bclocalduality}, the dimension of the prime ideal $\fp$ in $R$.
\end{rem}

\begin{rem}
Note that, in general, module categories over Noetherian ring spectra contain many non-trivial phantom maps, due to the failure of the generating hypothesis in these situations. For example, the analogue of the generating hypothesis is known to fail for the derived category of a ring~\cite{hlpgenhypdermod} or the stable module category of a finite group~\cite{ccmgenhypstmod} except in very simple cases. Since we require a splitting in the category of $R$-module spectra, the results of this subsection are crucial for the proof of the algebraic chromatic splitting hypothesis for Noetherian ring spectra.  
\end{rem}

\subsection{Proof of the main theorem}

Fix a prime ideal $\fp \in \Spec(\pi_*R)$. We start with two auxiliary lemmas. 

\begin{lem}\label{lem:derivedcompletion}
For any prime $\fp$ and $M \in \Mod_{R_{\fp}}^{\omega}$, there is a natural equivalence $\pi_*\Lambda_{\fp}M \cong (\pi_*M)_{\fp}^{\wedge}$, where $(-)_{\fp}^{\wedge}$ denotes (underived) $\fp$-completion. 
\end{lem}
\begin{proof}
By \cite[Prop.~3.19]{bhv2}, there exists a strongly convergent local homology spectral sequence of signature
\[
H_{*}^{\fp}(\pi_*M) \implies \pi_*\Lambda_{\fp}M.
\]
Recall that, on the $E_2$-page, the algebraic local homology $H_{*}^{\fp}$ computes the derived $\fp$-completion, see \cite[Thm.~2.5]{gm_localhomology} or \cite[Prop.~3.14]{bhv}. Since $M$ is assumed to be compact, $\pi_*M$ is a finitely generated $\pi_*R_{\fp}$-module, so this spectral sequence collapses to the desired isomorphism by the Artin--Rees lemma. 
\end{proof}

\begin{lem}\label{lem:phantommap}
The natural map $\pi_*\Delta_{\fp}M \to \pi_*L_{\cV(\fp)}M$ is zero for any compact $M \in \Mod_{R_{\fp}}^{\omega}$.
\end{lem}

\begin{proof}
By Krull's intersection theorem~\cite[Cor.~10.19]{am_commalg} and Lemma\nobreakspace \ref {lem:derivedcompletion}, the completion map $\pi_*M \to \pi_*\Lambda_{\fp}M$ is injective. Therefore, it follows from the long exact sequence associated to the fiber sequence $\Delta_{\fp}M \to M \to \Lambda_{\fp}M$ that the natural map $\pi_*\Delta_{\fp}M \to \pi_*M$ is zero. The pullback square of Lemma\nobreakspace \ref {lemp:pcfs} then yields the claim. 
\end{proof}

We are now ready to prove the main result of this section.

\begin{thm}\label{thm:chromaticsplitting}
The algebraic chromatic splitting hypothesis holds for any Noetherian $\bE_{\infty}$-ring spectrum $R$, i.e., for any  adjacent pair of primes $(\fp'\subset \fp)$, the natural map 
\[
\xymatrix{\iota_{M}\colon L_{\cV(\fp)}M \ar[r] & L_{\cV(\fp)}\Lambda_{\fp}M}
\] 
is split for all compact $M \in \Mod^{\omega}_{R_{\fp}}$ of type $\fp'$.
\end{thm}
\begin{proof}
By a thick subcategory argument it suffices to consider $M = R_{\fp}\mm\fp'$. We first claim that, for any adjacent pair $(\fp'\subset \fp)$, there are equivalences
\begin{equation}\label{eq:ldequiv}
L_{\cV(\fp)}R_{\fp}\mm\fp' \simeq \Lambda_{\fp'}R_{\fp'}\mm\fp' \simeq R_{\fp'}\mm\fp',
\end{equation}
where the last equivalence follows from the fracture square for $R_{\fp'}\mm \fp'$. 

 To this end, note that $R_{\fp} \mm \fp' \simeq L_{\cZ(\fp)}R \mm \fp' \simeq L_{\cZ(\fp)}\Gamma_{\cV(\fp')}R \mm \fp'$, where the last equivalence is a consequence of \cite[Lem.~5.11(2)]{benson_local_cohom_2008}. Applying \protect \MakeUppercase {C}orollary\nobreakspace \ref{cor:adjacentdecomp} to $R\mm\fp'$ then yields equivalences
\begin{align*}
L_{\cV(\fp)}\kos{R_{\fp}}{\fp'} &\simeq L_{\cV(\fp)}L_{\cZ(\fp)}\Gamma_{\cV(\fp')}R\mm\fp'\\
& \simeq \Gamma_{\fp'} R \mm \fp' \\
& \simeq R_{\fp'} \mm \fp'. 
\end{align*}

We now claim that the map $\Delta_{\fp}R_{\fp}\mm\fp' \to L_{\cV(\fp)}R_{\fp}\mm\fp'$ is phantom. Indeed, let $C \to \Delta_{\fp}R_{\fp}\mm\fp'$ be a map from a compact module $C$. The composite $C \to \Delta_{\fp}R_{\fp}\mm\fp' \to L_{\cV(\fp)}R_{\fp}\mm\fp'$ is adjoint to a homotopy class $R \to \Delta_{\fp}C_{\fp}^{\vee}\mm\fp' \to L_{\cV(\fp)}C_{\fp}^{\vee}\mm\fp'$, which must be trivial by Lemma\nobreakspace \ref {lem:phantommap}. By \protect \MakeUppercase {C}orollary\nobreakspace \ref {cor:nophantom} there are no non-trivial phantom maps with target $R_{\fp'}\mm \fp'$, hence $\Delta_{\fp}R_{\fp}\mm \fp' \to L_{\cV(\fp)}R_{\fp}\mm \fp' \simeq R_{\fp'}\mm\fp'$ is zero, so $\iota_{R_{\fp}\mm \fp'}$ is split. 
\end{proof}

\begin{rem}
This argument has been inspired by an approach to the original chromatic splitting conjecture appearing in unpublished work by Sadofsky. 
\end{rem}

In the following we will consider natural maps
\[
\xymatrix{
\iota_n \colon R_{\fp'}\mm(\fp')^{(n)} \otimes M \ar[r] & R_{\fp'}\mm(\fp')^{(n)} \otimes \Lambda_{\fp}M
}
\]
for $M \in \Mod_{R_{\fp}}^{\omega}$, and $(\fp' \subset \fp)$ a pair of adjacent primes. These maps are split by \protect \MakeUppercase {T}heorem\nobreakspace \ref {thm:chromaticsplitting}, and we say that the splitting can be chosen compatibly if it is compatible with the maps in the associated tower obtained by varying $n$.  

The next two results give integral versions of  \protect \MakeUppercase {T}heorem\nobreakspace \ref {thm:chromaticsplitting}.

\begin{cor}\label{cor:localhomsplitting}
For any pair $(\fp' \subset \fp)$ of adjacent primes and $M \in \Mod_{R_{\fp}}^{\omega}$, if the splittings appearing in the previous theorem can be chosen compatibly, then the natural map
\[
\xymatrix{\widetilde \Lambda_{\fp'}M \ar[r] & \widetilde \Lambda_{\fp'}\Lambda_{\fp}M}
\]
is split.
\end{cor}
\begin{proof}
Using the fact that $L_{\cV(\fp)}$ is smashing, as well as an argument similar to that given in proving \eqref{eq:ldequiv}, one has that $R_{\fp'}\mm(\fp')^{(n)} \otimes M \simeq L_{\cV(\fp)}(M \otimes R_{\fp}\mm(\fp')^{(n)})$. This implies that, for $M$ as in the corollary and any $n\ge 0$, there is a natural map 
\[
\xymatrix@C=1.3em{R_{\fp'}\mm(\fp')^{(n)} \otimes M \simeq L_{\cV(\fp)}(M \otimes R_{\fp}\mm(\fp')^{(n)}) \ar[r]^-{\iota_n} & L_{\cV(\fp)}\Lambda_{\fp}(M \otimes R_{\fp}\mm(\fp')^{(n)}) \simeq R_{\fp'}\mm(\fp')^{(n)} \otimes \Lambda_{\fp}M,}
\]
where we have used the fact that $R_{\fp}\mm(\fp')^{(n)}$ is compact in $\Mod_{R_{\fp}}$, so that tensoring with $R_{\fp}\mm(\fp')^{(n)}$ commutes with $\Lambda_{\fp}$.

Local duality (\protect \MakeUppercase {T}heorem\nobreakspace \ref {thm:localduality}) and Lemma\nobreakspace \ref {lem:koszulformula} then give equivalences
\begin{equation}\label{eq:limitformula}
 \psi_N  \colon \widetilde \Lambda_{\fq}N \simeq \lim_{n} \Sigma^{nd} R_{\fq}\mm \fq^{(n)}\otimes N
\end{equation}
for any prime ideal $\fq$ and $N \in \Mod_R$. This leads to the following commutative diagram:
\[
\xymatrix{\widetilde \Lambda_{\fp'}M \ar[r]^-{\kappa} \ar[d]_{\psi_M}^{\sim} & \widetilde \Lambda_{\fp'}\Lambda_{\fp}M \ar[d]_{\sim}^{\psi_{\Lambda_{\fp}M}} \\
\lim_n \Sigma^{nd}R_{\fp'}\mm(\fp')^{(n)} \otimes M \ar[r]_-{\iota} & \lim_n \Sigma^{nd}R_{\fp'}\mm(\fp')^{(n)} \otimes \Lambda_{\fp}M.}
\]
The map $\iota$ is split because it is a limit of split maps by \protect \MakeUppercase {T}heorem\nobreakspace \ref {thm:chromaticsplitting} which are assumed to be compatible with the inverse system.
\end{proof}

We now exhibit a condition on $R$ such that the map of Corollary\nobreakspace \ref{cor:localhomsplitting} is split, which covers many examples of interest. We note that the following corollary simply proves the existence of some splitting; we do not claim that it is the limit of the individual splittings. 
\begin{cor}\label{cor:splitting2}
Assume that the ring spectrum $R_{\fq}$ is absolute Gorenstein for any $\fq \in \Spec(\pi_*R)$. For any pair $(\fp' \subset \fp)$ of adjacent primes and $M \in \Mod_{R_{\fp}}^{\omega}$, the natural map
\[
\xymatrix{\widetilde \Lambda_{\fp'}M \ar[r] & \widetilde \Lambda_{\fp'}\Lambda_{\fp}M}
\]
is split. 
\end{cor}
\begin{proof}
First, note that it is enough to prove the claim for $M=R$. Applying $\widetilde{\Lambda}_{\fp'}$ to the fiber sequence $\Delta_{\fp}R_{\fp} \to R_{\fp} \to \Lambda_{\fp}R_{\fp}$, it then suffices by Corollary\nobreakspace \ref{cor:gorensteinphantom} to show that the map $\phi\colon\widetilde{\Lambda}_{\fp'}\Delta_{\fp}R_{\fp} \to \widetilde{\Lambda}_{\fp'}R_{\fp}\simeq \widetilde{\Lambda}_{\fp'}R$ is phantom. The same argument as the one given at the end of the proof of Theorem\nobreakspace \ref{thm:chromaticsplitting} reduces this to the claim that
\[
\xymatrix{\phi_*\colon \pi_*\widetilde{\Lambda}_{\fp'}\Delta_{\fp}C \ar[r] & \pi_*\widetilde{\Lambda}_{\fp'}C}
\]
is zero for all $C \in \Mod_{R_{\fp}}^{\omega}$. We will show this for $C=R_{\fp}$, the general case being an easy consequence. To this end, we use the limit description of \eqref{eq:limitformula}.

As seen in the proof of Proposition\nobreakspace \ref{prop:bcduality}, the graded module $\pi_*(R_{\fp'}\mm\fp')$ is finitely generated over $(\pi_*R_{\fp'})/\fp' \cong (\pi_*R)_{\fp'}/\fp'$, so we see inductively that $\pi_*(R_{\fp'}\mm\fp'^{(n)})$ is a degreewise finite-dimensional vector space over $(\pi_*R_{\fp'})/\fp'$ for all $n \ge 1$. Therefore, the tower $(\pi_*(\Sigma^{nd}R_{\fp'}\mm\fp'^{(n)}))_n$ satisfies the Mittag-Leffler condition, hence $\lim_n^1\pi_*(\Sigma^{nd}R_{\fp'}\mm\fp'^{(n)}) = 0$. The associated Milnor sequences thus give a commutative diagram
\[
\xymatrix{\pi_*\widetilde{\Lambda}_{\fp'}\Delta_{\fp}R_{\fp}  \cong \pi_*\lim_n(\Sigma^{nd}R_{\fp'}\mm\fp'^{(n)} \otimes \Delta_{\fp}R_{\fp}) \ar[r] \ar[d]_{\phi_*} & \lim_n\pi_{*}(\Sigma^{nd}R_{\fp'}\mm\fp'^{(n)} \otimes \Delta_{\fp}R_{\fp}) \ar[r] \ar[d] & 0\\
\pi_*\widetilde{\Lambda}_{\fp'}R_\fp \cong \pi_*\lim_n \Sigma^{nd}R_{\fp'}\mm\fp'^{(n)} \ar[r]_-{\sim} & \lim_n\pi_{*}(\Sigma^{nd}R_{\fp'}\mm\fp'^{(n)}) \ar[r] & 0.}
\]
The right vertical arrow is zero because each of the individual maps
\[
\xymatrix{\pi_{*}(R_{\fp'}\mm\fp'^{(n)} \otimes \Delta_{\fp}R_{\fp}) \ar[r] & \pi_{*}(R_{\fp'}\mm\fp'^{(n)})}
\]
is zero as was shown in the proof of Theorem\nobreakspace \ref{thm:chromaticsplitting}. This implies $\phi_* = 0$. 
\end{proof}

Observe that variants of \protect \MakeUppercase {T}heorem\nobreakspace \ref {thm:chromaticsplitting} and \protect \MakeUppercase {C}orollaries \nobreakspace \ref{cor:localhomsplitting} and \ref{cor:splitting2} hold without the $\fp$-locality assumption on $M$ if we use $\widetilde{\Lambda}$ in place of $\Lambda$, see \protect \MakeUppercase {R}emark\nobreakspace \ref {rem:localhomvariant}. We decided to give the version stated above since it is consistent with the work of Benson, Iyengar, and Krause.

\subsection{Applications and examples}

The algebraic chromatic splitting hypothesis is particularly relevant in categories $\Mod_R$ which are stratified by $\pi_*R$ in the sense of Benson, Iyengar, and Krause. We recall from~\cite{benson_stratifying_2011} that this means that the following two conditions hold:
\begin{enumerate}
	\item The category $\Mod_R$ satisfies the local to global principle. 
	\item The localizing subcategories $\Gamma_{\fp}\Mod_{R} = \Mod_{R_{\fp}}^{\fp-\text{tors}}$ are minimal for all $\fp \in \Spec(\pi_*R)$. 
\end{enumerate} 
In this case, the localizing subcategories of $\Mod_R$ are in bijection with subsets of $\Spec(\pi_*R)$ and are detected by the functors $\Gamma_{\fp}$. Informally speaking, the splitting of \protect \MakeUppercase {T}heorem\nobreakspace \ref {thm:chromaticsplitting} thus describes how $\Mod_R$ can be assembled from its indecomposable pieces $\Gamma_{\fp}\Mod_{R}$. In this subsection, we review three known classes of stratified categories $\Mod_R$ with $R$ Noetherian and give one explicit example. 

\begin{ex}
The easiest case is when $R$ is a discrete Noetherian commutative ring. It is known by work of Neeman~\cite{neemanchromtower} that the localizing subcategories of the derived category $\Mod_R$ of $R$-modules are classified by subsets of $\Spec(R)$. From this, it immediately follows that $\Mod_R$ is stratified by $R$, so the algebraic chromatic splitting hypothesis holds. This complements Neeman's work on the nilpotence theorem and the telescope conjecture by resolving the last remaining algebraic analogue of a prominent chromatic conjecture for $\Mod_R$. 
\end{ex}

The next example contains the previous one as a special case.

\begin{ex}
Suppose $R$ is an evenly concentrated Noetherian $\mathbb{E}_{\infty}$-ring spectrum with $\pi_*R$ regular. Dell'Ambrogio and Stanley show in \cite{affineweaklyregular} that the localizing subcategories of $\Mod_R$ are in bijection with subsets of $\Spec(\pi_*R)$, from which the stratification result can be deduced. 

As a concrete example, we can take $R = E_n$ to be Morava $E$-theory, a ring spectrum with coefficients $\bW\F_{p^n}\llbracket u_1,\ldots,u_{n-1} \rrbracket[u^{\pm 1}]$, where $\bW\F_{p^n}$ denotes the ring of Witt vectors on $\F_{p^n}$, the $u_i$ are in degree 0, and $u$ is in degree $-2$. By a theorem of Goerss and Hopkins \cite{goersshopkins_etheory}, $E_n$ admits the structure of an $\bE_{\infty}$-ring spectrum and thus satisfies the above conditions. 

Our results in this case should be contrasted with Devinatz's counterexample to the $BP$-analogue of the chromatic splitting conjecture, see~\cite{devinatz_counterex}. Recall that the Brown--Peterson spectrum $BP$ has coherent (but not Noetherian) graded commutative homotopy ring $\pi_*BP \cong \Z_{(p)}[v_1,v_2,\ldots]$, so that $I_n = (p,v_1,\ldots,v_{n-1})$ is a regular prime ideal in $BP_*$.  By \cite[Thm.~7.23]{bhv} and \cite[Prop.~3.7]{bs_centralizers} combined with Lemma\nobreakspace \ref {lem:koszulformula}, the chromatic localization functors $L_{n-1}$ and $L_{K(n)}$ restricted to $\Mod_{BP}$ are equivalent to the algebraic localization functors $L_{\cV(I_n)}$ and $\Lambda_{\cV(I_n)}L_{\cV(I_n)}$, respectively (see also \cite[Prop.~3.12]{bhv2}). We thus obtain a commutative square
\[
\xymatrix{L_{n-1}BP_p \ar[r] \ar[d]_{\sim} & L_{n-1}L_{K(n)}BP_p \ar[d]^{\sim} \\
L_{\cV(I_n)}BP_p \ar[r]_-{\iota} & L_{\cV(I_n)}\Lambda_{\cV(I_n)}L_{\cV(I_n)}BP_p}
\]
in which the top map is the one considered by Devinatz, whereas the bottom one is a version of the map $\iota$ of Lemma\nobreakspace \ref {lemp:pcfs} applied to pair of primes $((0) \subseteq I_n)$ and evaluated on the $p$-completion of $BP$. In this situation, Devinatz proves that $\iota$ is not a split injection whenever $n\ge 2$, i.e., as soon as the pair of primes $((0) \subseteq I_n)$ is not consecutive. 

This highlights the subtlety of the algebraic chromatic splitting hypothesis. 
\end{ex}

We now provide a motivating example for studying the algebraic chromatic splitting hypothesis for ring spectra.

\begin{ex}\label{ex:groups}
Let $C^*(BG,k)\simeq k^{hG}$ be the $\bE_{\infty}$-ring of cochains on the classifying space of a compact Lie group $G$ with $\pi_0G$ a finite $p$-group and with coefficients in a field $k$. Benson and Greenlees \cite{bg_stratifyingcompactlie}, extending earlier work of Benson, Iyengar, and Krause \cite{bik_finitegroups}, prove that $\Mod_{C^*(BG,k)}$ is stratified by the canonical action of $\pi_{-*}C^*(BG,k) \cong H^*(G,k)$. Assume further that $G$ is finite. Since $H^*(G,k)$ is Noetherian by a theorem due to Evens \cite{evens_cohomfinitegroups} and Venkov \cite{venkov_cohomfinitegroups}, and $C^*(BG,k)_{\fq}$ is absolute Gorenstein for any $\fq \in \Spec(H^*(G,k))$ by work of Benson and Greenlees, see \cite[Prop.~4.33]{bhv2}, \protect \MakeUppercase {C}orollary\nobreakspace \ref {cor:splitting2} applies to give a natural splitting of the map
\[
\xymatrix{\widetilde \Lambda_{\fp'}k_{\fp} \ar[r] & \widetilde\Lambda_{\fp'}\Lambda_{\fp}k_{\fp}}
\]
for any adjacent pair of primes $\fp',\fp \in \Spec(H^*(G,k))$. Note that, for $G$ a finite group, the local homology $\Lambda_{\fp}k$ is in some sense dual to the Rickard idempotents, see \cite[Sec.~10]{benson_local_cohom_2008}. 

Using the techniques of \cite{bhv2} and \cite{krausebenson_kg}, the chromatic splitting can be transported to the stable module category $\StMod_{kG}$ for any finite $p$-group $G$. Therefore, we obtain a qualitative description of how the stable module module category is locally built up from its indecomposable layers, at least for compact objects. 
\end{ex}

Finally, we conclude with an example which illustrates our results in the case of an explicit Noetherian $\bE_{\infty}$-ring spectrum. 

\begin{ex}
Let $G = \Z/2 \times \Z/2 = \langle g_1,g_2 \rangle$ be the Klein group and $k$ an algebraically closed field of characteristic $2$. The $\bE_{\infty}$-ring spectrum $C^*(BG,k)$ of cochains on $G$ with coefficients in $k$ has homotopy groups
\[
\pi_{-*}C^*(BG,k) \cong H^*(G,k) \cong k[\zeta_1,\zeta_2]
\]
with $\zeta_1$ and $\zeta_2$ generators in degree 1. Therefore, $C^*(BG,k)$ is Noetherian and there is an equivalence
\[
\Mod_{C^*(BG,k)} \simeq \Stable_{kG}
\]
as $G$ is a 2-group, where $\Stable_{kG}$ is the slight enlargement of the stable module category $\StMod_{kG}$ constructed by Benson and Krause~\cite{krausebenson_kg}. We may therefore work within the stable module category $\StMod_{kG}$; in particular, all our constructions can be restricted to this subcategory of $\Stable_{kG}$. Consider the prime ideal $\fp = (\zeta_1)$ generated by $\zeta_1$ and the consecutive pair of primes $((0) \subseteq (\zeta_1))$ in $\Spec(H^*(G,k))$. Any compact $M \in \Stable_{kG}$ is of type $0$ so we may take $M = k$. By \eqref{eq:ldequiv} and \cite[Sec.~4.3.2]{bik_mfo}, we obtain equivalences
\[
L_{\cV(\zeta_1)}k_{\zeta_1}  \simeq L_{\cZ(0)}k \simeq k(t) \oplus k(t),
\]
where $k_{\zeta_1} = L_{\cZ(\zeta_1)}k$ has cohomology $H^*(G,k_{\zeta_1}) \cong k[\zeta_1,\zeta_2]_{(\zeta_1)}$ and the two generators of $G$ act via the matrices
\[
g_1\mapsto \begin{pmatrix}
\Id & 0 \\
\Id & \Id
\end{pmatrix} 
\qquad 
g_2\mapsto \begin{pmatrix}
\Id & 0 \\
t& \Id
\end{pmatrix}.
\]
The fracture square of Lemma\nobreakspace \ref {lemp:pcfs} then takes the following form:
\[
\xymatrix{
k_{\zeta_1} \ar[d] \ar[r] & \Lambda_{\zeta_1}k_{\zeta_1} \ar[d] \\
k(t) \oplus k(t) \ar[r]_-{\iota_{\zeta_1}} & L_{\cV(\zeta_1)}\Lambda_{\zeta_1}k_{\zeta_1}.}
\]
We do not know of an elementary description of the objects in the right column of this square, but our main theorem~\protect \MakeUppercase {T}heorem\nobreakspace \ref {thm:chromaticsplitting} implies that $\iota_{\zeta_1}$ is a split injection. This provides a construction of $k_{\zeta_1}$ from the three remaining pieces with smaller support. 
\end{ex}

\biblio
\bibliography{duality}\bibliographystyle{alpha}
\end{document}